\documentclass{article}
\usepackage[utf8]{inputenc}

\usepackage{amsmath}
\usepackage{mathtools}
\usepackage{amssymb}
\usepackage{amsthm}
\usepackage{graphicx}
\usepackage{amscd}
\usepackage{epic, eepic}
\usepackage{url}
\usepackage{color}
\usepackage[utf8]{inputenc} 
\usepackage{makebox}
\usepackage{comment}
\usepackage{enumerate}   
\usepackage{hyperref}
\usepackage[]{todonotes}
\theoremstyle{plain}
\newtheorem{theorem}{\bf Theorem}[section]
\newtheorem{lemma}[theorem]{\bf Lemma}

\newtheorem{cor}[theorem]{\bf Corollary}

\newtheorem{problem}[theorem]{\bf Problem}
\newtheorem{prop}[theorem]{\bf Proposition}

\newtheorem{nota}[theorem]{\bf Notation}

\newtheorem{remark}[theorem]{\bf Remark}
\newtheorem{defi}[theorem]{\bf Definition}
\newtheorem{result}[theorem]{\bf Result}

\def\ex{\hbox{\rm ex}}

\def\EX{\hbox{\rm EX}}

\title{Multicolor Turán numbers}

\author{András Imolay\thanks{Eötvös Loránd University, Budapest, Hungary. The author is partially supported by NKFIH grant KH130371 and partially supported by the ÚNKP-21-2 New National Excellence Program of the Ministry for Innovation and
Technology from the source of the National Research, Development and Innovation Fund.
E-mail: {\tt imolay.andras@gmail.com}}
\and 
János Karl \thanks{Eötvös Loránd University, Budapest, Hungary.. Email: {\tt karlj9836@gmail.com }}
\and Zolt\'an L\'or\'ant Nagy\thanks{ELKH--ELTE Geometric and Algebraic Combinatorics Research Group,
  E\"otv\"os Lor\'and University, Budapest, Hungary. The author is supported by the Hungarian Research Grant (NKFI) No. K  134953.  	E-mail: {\tt nagyzoli@cs.elte.hu}}
\and  Benedek Váli \thanks{Cambridge University, UK.
E-mail: {\tt benedekvali@gmail.com }}
}
\date{}

\begin{document}

\maketitle

\begin{abstract} We consider a natural generalisation of  Turán's forbidden subgraph problem and the Ruzsa-Szemerédi problem by studying the maximum number $\ex_F(n,G)$ of edge-disjoint copies of a fixed graph $F$ can be placed on an $n$-vertex ground set without forming  a subgraph $G$ whose edges are from different $F$-copies. We determine the pairs $\{F, G\}$ for which the order of magnitude of $\ex_F(n,G)$ is quadratic and prove several asymptotic results using various tools from the regularity lemma and supersaturation to graph packing results. 
    
\end{abstract}

\section{Introduction}

In this paper we address a problem which is a generalisation of Turán-type problems. For a fixed pair of nonempty graphs $F$ and $G$, determine the maximum number of edge-disjoint copies of $F$ on $n$ vertices which does not form a subgraph $G$ whose edges are from different $F$-copies. Colouring each $F$-copy with a colour of its own, the prohibited configuration is a multicolor $G$. This maximum is denoted by $\ex_F(n,G)$ and we refer to it as {\em $F$-multicolor Turán number of $G$} or simply multicolor Turán number of $G$ if $F$ is clear from the context. In this paper we assume $F$ and $G$ have no isolated vertices.\\
Observe that the case when $F$ is an edge ($F=K_2$) is the ordinary Turán problem. For general
background of the forbidden subgraph problem, the reader is referred to the excellent surveys of Füredi and Simonovits \cite{survey} and of Keevash \cite{Keevash-survey}.

Several coloured variants of the forbidden subgraph problem have been raised recently. Alon, Jiang, Miller and Pritikin \cite{Alon_Ramsey} introduced the
problem of finding a rainbow copy of a graph $G$ in an edge colouring of the complete graph on $n$ vertices in which each
colour appears at most $m$ times at each vertex. 
Determining the maximum
number of edges in a graph on $n$ vertices that has a proper edge-colouring with no
rainbow $H$ is the so called rainbow Turán problem, initiated by Keevash, Mubayi, Sudakov and Versta\"ete \cite{Keevash}. 
Keevash, Saks, Sudakov and Versta\"ete investigated the maximum number of edges in a multigraph on $n$ vertices,
that has a simple $k$-edge-colouring not containing a multicoloured copy of $H$ in \cite{KSSV}. Another variant was studied by Conlon and Tyomkyn \cite{CT} where the number of colours is to be minimized in a properly colored complete graph, subject to the property that it does not contain a given number of copies of a fixed graph $F$ with the same colouring pattern.

There are numerous motivations to initiate the study of our variant. While it provides a generalisation to the classical Turán-type theory, it also serves as an example for extremal type coloring problems, which has a long and rich history. One might notice that the exact or asymptotic results depend heavily on the existence of (almost) $F$-factors in extremal graphs avoiding $G$. This highlights the connection to yet another extremal graph theoretic field. One of the most important reasons,  as it will be pointed out later on,  is that this variant  includes  hypergraph Turán problems as well, notably e.g. the  celebrated Ruzsa-Szemerédi $(6,3)$ problem \cite{Ruzsa} and its generalisations. This indicates  the possible applications of results of number theory and applications to number theory, too. Another connection to hypergraph extremal problems follows from the case when $F$ and $G$ are cliques. In particular, the minimum size of independent sets a Steiner triple system STS(n) can have gives exact bounds for $\ex_{K_3}(n,K_r)$ when $r$ is large enough. Finally,  results for certain  pairs of subgraphs $F, G$ may be applied in other extremal problems and  yield improvements in the field of forbidden subgraph problems or generalized Turán-type problems. This connection will be discussed in  section 5.

\bigskip

We formalise the main problem and continue by discussing the main results of the paper.

\begin{nota}
\hfill

$\ex_F(n,G)=\max\{k \ | \ \dot\bigcup_{i=1}^k F_i \subseteq K_n, F_i\sim F, \not\exists G\subseteq \bigcup_{i=1}^k F_i, \  \mbox{s.t.} \ |E(G)\cap E(F_i)|\leq 1  \}.$
\end{nota}
\medskip
The goal  is to determine the order of magnitude of this function for given pairs of graphs and to obtain asymptotic or even sharp results if it is possible.
It is fairly easy to see that $\ex_F(n,G)$ is bounded from above by the Turán number $\ex(n,G)$ for each nonempty graph $F$. We will prove that $\ex_F(n,G)$ is the same order of magnitude as 
$\ex(n,G)$ provided that $\liminf\frac{\ex(n,G)}{\ex(n,F)}>1$, which determines the order of magnitude for a large class of pairs.

Our main result in this direction is as follows.

\begin{theorem}\label{mainTHM}
$\ex_F(n,G)=\Theta(n^2)$ if and only if there is no homomorphism from $G$ to $F$, and $\ex_F(n,G)=o(n^2)$ otherwise.
\end{theorem}

We also obtained asymptotic results.

\begin{theorem} \label{mainth}
If $\chi(F)<\chi(G)$ then
$$\ex_F(n,G) \sim \frac{\ex(n,G)}{e(F)} \sim \frac{1-\frac{1}{\chi(G)-1}}{2e(F)}n^2.$$
\end{theorem}

Next we show the correspondence between the famous Ruzsa-Szemerédi (6,3)-problem and a particular case of the multicolor Turán problem. The problem of Ruzsa and Szemerédi known as the   (6,3)-problem asks for the maximum number of edges in an $n$-vertex graph in which every edge belongs to a unique triangle. Almost equivalently, we want to determine the maximum number of hyperedges in a $3$-uniform  hypergraph on $n$ vertices in which no six vertices span three or more hyperedges.\footnote{Apart from a linear component in which triples can be assigned to disjoint triangles, some components on four vertices with two triplets may occur.} 
Observe that such an extremal hypergraph must be  linear for large enough $n$, and then the forbidden configuration is  the same as in the multicolor Turán problem for $F=G=K_3$. Namely, the configuration of three triplets on six vertices so that each pair of triplets shares a pairwise distinct common vertex.  Answering a question of Brown, Erdős and Sós, Ruzsa and Szemerédi obtained the following result in our terminology.

\begin{theorem}[Ruzsa, Szemerédi \cite{Ruzsa}]

$$n^2\exp(-O(\sqrt{\log{n}}))<\ex_{K_3}(n,K_3)=o(n^2).$$
\end{theorem}

For further details of the proof and recent generalisations we refer to  \cite{Gowers-Janzer}.

The paper is organised as follows. In Section 2. we set up the main notations and prove some general preliminary results concerning the multicolor Turán function. We also introduce the key tools for the proof of Theorem \ref{mainTHM} which includes  the celebrated Regularity Lemma. We continue by proving Theorem \ref{mainTHM} and \ref{mainth} in Section 3. Then we establish some bounds for specific pairs of graphs in  Section 4. We point out and discuss the strong  relation between the so called linear (hypergraph) Turán problems and multicolor Turán numbers, namely that $\ex_r^{lin}(n,G)=\ex_{K_r}(n,G)$ holds, where  the left hand side denotes the linear Turán number of a Berge-$G$, i.e.  the maximum number of hyperedges in a  Berge-$G$-free $r$-uniform linear hypergraph on $n$ vertices. We conclude our paper by pointing out possible applications and posing some open problems.

\section{Preliminaries}

\subsection{Notations} In this paper we only consider simple graphs. For a graph $G$, $e(G)$, $v(G)$, $\alpha(G)$ and $\chi(G)$ will denote the number of edges, vertices, the size of the largest independent set and the chromatic number of $G$, respectively. The lower index denotes the number of vertices of a graph: $P_n$, $C_n$, $K_n$ denote the path, cycle and complete graph on $n$ vertices, $K_{a,b}$ the biclique on $a+b$ vertices and $T_{n,m}$ the Turán graph with $n$ vertices and $m$ classes. $\EX(n,G)$ denotes the set of graphs on $n$ vertices with $ex(n,G)$ edges which have no subgraph isomorphic to $G$.

For a graph $G=(V,E)$ we will use the notation $e(X,Y)$ for the number of edges between $X$ and $Y$, where $X$ and $Y$ are disjoint subsets of $V$. We denote the density (Definition~\ref{density}) of $X$ and $Y$ with $d(X,Y)$. We denote with $\deg_G(v)$ the degree of the vertex $v$.

Next we define the blow-up of a graph.
Let $R$ be a graph and $t$ a positive integer. Define $R[t]$ as the graph obtained from $R$ such that we replace every vertex $v$ of $R$ with $t$ vertices, call this set $A_v$, and $xy \in E(R[t])$ if and only if $uv \in E(R)$ where $x \in A_u$ and $y \in A_v$. We call $R[t]$ the $t$-blow-up of $R$.

We will need the theorem of  Haxell and Rödl in connection with graph packings  \cite{packing}. We use the same notations as in the mentioned article, which we summarize here.
Let $F$ and $G$ be two graphs. We denote with ${G \choose H}$ the set of all subgraphs of $G$ isomorphic to $H$. We call a subset $\mathcal{A}$ of ${G \choose H}$ an $H$-packing if any two graphs $A,B \in \mathcal{A}$ are edge-disjoint. Let $\nu_H(G)$ be the size of the largest $H$-packing. A fractional $H$-packing is a function $\psi: {G \choose H} \to [0,1]$ such that for each edge $e \in E(G)$ we have $\sum_{e \in H' \in {G \choose H}} \psi(H') \leq 1$. Let $|\psi|=\sum_{H' \in {G \choose H}} \psi(H')$. We say that a fractional $H$-packing is maximal if $|\psi|$ is maximal and  denote this maximum by $\nu_H^*(G).$

For discussing the possible application to a generalized Turán problem in section 5., we need the notion of $\ex(n,H,G)$ which denotes the maximum number of copies of $H$ in $G$-free graphs on $n$ vertices.

\subsection{Basic results}

Observe the following trivial monotonicity results of multicolor Turán numbers.

\begin{prop}\label{monotonicity}
We have monotonicity in $F$, $G$ and $n$ as follows.
\begin{enumerate}[(1)]
    \item If $F$ is a subgraph of $F'$ and $G$ is any graph then $$\ex_{F'}(n,G) \leq \ex_{F}(n,G).$$
    \item If $G_1$ is a subgraph of $G_2$ and $F$ is any graph then $$\ex_F(n,G_1) \leq \ex_F(n,G_2).$$
    \item If $F$ and $G$ are graphs and $n_1<n_2$ then $$\ex_F(n_1,G) \leq \ex_F(n_2,G).$$
\end{enumerate}
\end{prop}

Putting a single edge into $F'$ in (1) shows the connection of the multicolor Turán numbers and the original Turán numbers.

\begin{cor}
For any graphs $F$ and $G$ we have $$\ex_F(n,G)\leq\ex(n,G).$$
\end{cor}

\begin{prop} \label{simplelower}
$\ex_F(n,G)\geq\frac1{e(F)}\left(\ex(n,G)-\ex(n,F)\right)$.
\end{prop}

\begin{proof}
Take $H\in\EX(n,G)$, and pack a maximal amount of disjoint $F$-copies into it. Then there must be at most $\ex(n,F)$ many edges not in any of the copies, otherwise we could pack another $F$-copy. This gives a construction with at least the required number of $F$-copies.
\end{proof}

\subsection{Important Lemmas}

In this subsection we summarize the main lemmas and theorems we will apply.

We will need the following supersaturation result of Erdős and Simonovits \cite{E-S_supersat}.

\begin{theorem}[Erdős--Simonovits] \label{supersat}
Let $G$ be a  graph with chromatic number $\chi(G)=r$. Then for any $c>0$ there exists a $c'>0$ such that if $\mathcal{G}_{n}$ is a graph on $n$ vertices with $e(\mathcal{G}_{n})>\left(1-\frac{1}{r-1}+c\right)\displaystyle\binom{n}2$ then $\mathcal{G}_{n}$ contains at least $c'n^{v(G)}$ copies of $G$. 
\end{theorem}

We will also use a theorem due to Haxell and Rödl  \cite{packing}.

\begin{theorem}[Haxell--Rödl] \label{HR}
Let $H$ be a fixed graph and let an $\eta>0$ be given. Then there exists an integer $N$ such that for any graph $G$ on $n \geq N$ vertices we have
$$\nu_H^*(G)-\nu_H(G) \leq \eta n^2.$$
\end{theorem}

The Szemerédi regularity lemma will be  one of the most important tools in our proof. We use the notations and theorems from \cite{Komlos} but we state the most relevant definitions and results that we will use. 

\begin{defi}[Density]\label{density}
Let $\mathcal{G}=(V,E)$ be a graph and $X,Y \subset V$ are disjoint subsets. We define the density of $X$ and $Y$ with the following formula.
$$d(X,Y)=\frac{e(X,Y)}{|X| \cdot |Y|}.$$
\end{defi}


\begin{defi}[Regularity condition]
Let $\varepsilon>0$. Given a graph $\mathcal{G}=(V,E)$ and $A,B \subset V$ are disjoint subsets of the vertex set. We say that the pair $(A,B)$ is $\varepsilon$-regular, if for every $X \subset A$ and $Y \subset B$ with $|X|>\varepsilon |A|$ and $|Y|>\varepsilon |B|$ we have
$$|d(X,Y)-d(A,B)|<\varepsilon.$$
\end{defi}

Now we are ready to state the famous Szemerédi regularity lemma.  

\begin{theorem}[Degree form of Szemerédi regularity lemma]
For all $\varepsilon>0$ there exists an $M$ such that for all graph $\mathcal{G}$ and $0 \leq d \leq 1$ there is a partition of $V(\mathcal{G})$ to $V_0, V_1, \ldots, V_k$ and there exists a $\mathcal{G}' \subset \mathcal{G}$ with the following properties. 
\begin{enumerate}[(1)]
    \item $k \leq M,$ 
\item $|V_0| \leq \varepsilon |V|,$
$|V_i|=m$ for all $1 \leq i \leq k$ for some constant $m$, 
\item $\deg_{\mathcal{G}'}(v)>\deg_G(v)-(d+\varepsilon)|V|$ for all vertex $v \in V(\mathcal{G})$, 
\item $e(\mathcal{G}'(V_i))=0$ for all $i \geq 1$, 
and for all pair $1 \leq i<j \leq k$ $\mathcal{G}'(V_i,V_j)$ is $\varepsilon$-regular with density $0$ or at least $d$. 
\end{enumerate}
\end{theorem}

It is easy to see (and it is also in \cite{Komlos}) that as a consequence we get that if we define $\mathcal{G}''=\mathcal{G}'-V_0$ then $e(\mathcal{G}'')>e(\mathcal{G})-(d+3\varepsilon)\frac{n^2}{2}$.

We will also apply the so-called key lemma which often appears when the regularity lemma is used.

\begin{lemma}[Key lemma]
Given $d>\varepsilon>0$, a graph $R$ and an integer $m>0$. Let $\mathcal{G}$ be a graph obtained from $R$ such that we replace every vertex $v$ with a set $A_v$ of $m$ vertices and for each edge $uv \in E(R)$ choose the edges such that $(A_u, A_v)$ is $\varepsilon$-regular with density at least $d$, and there are no more edges of $\mathcal{G}$, so $\mathcal{G} \subset R[m]$ where $R[m]$ is the $m$-blow-up of $R$. \\
Let $H$ be a graph of maximal degree $\Delta>0$ for which $H \subset R[t]$ for some constant $t$. Let 
$$\varepsilon_0=\frac{(d-\varepsilon){^\Delta}}{2+\Delta}.$$
Assume that $\varepsilon \leq \varepsilon_0$ and $t-1 \leq \varepsilon_0 m$. Then the number of subgraphs isomorphic to $H$ in $\mathcal{G}$ is at least
$(\varepsilon_0 m)^{v(H)}$.
\end{lemma}

\section{Proof of the main theorems}

\begin{proof} [Proof of Theorem~\ref{mainth}] We wish to show that if $\chi(F)<\chi(G)$ then
$$\ex_F(n,G) \sim \frac{\ex(n,G)}{e(F)} \sim \frac{1-\frac{1}{\chi(G)-1}}{2e(F)}n^2.$$
Note first that the second asymptotic  equivalence holds by the Erdős-Stone-Simonovits theorem \cite{Erdos-Stone}, so we only need to prove the first one.

The idea of the proof is to take an extremal construction  from $\EX(n,G)$. For the lower bound we do an $F$-packing on this extremal graph while we use supersaturation to prove the upper bound.

We begin with the lower bound.
Let $r=\chi(G)$. Let $\varepsilon>0$ be given. From the asymptotic structure theorem of Erdős and Simonovits \cite{survey} we know that the Turán graph $T_{n,r-1}$ is close (in edit distance) to the extremal graph  on $n$ vertices without  $G$ as a subgraph. If the multipartite graph $T_{n,r-1}$ does not have equal size partition classes then we delete one vertex from each class which has more than $\frac{n}{r-1}$ vertices, together with the incident edges.

 Let us denote this graph with $\mathcal{G}_{n'}$. We deleted less than $n(r-1)$ edges, so $e(\mathcal{G}_{n'})>\ex(n,G)-~\varepsilon n^2$ holds when $n$ is large enough by the Erdős-Stone-Simonovits theorem \cite{Erdos-Stone}.

Now we want to prove that we can almost pack $\mathcal{G}_{n'}$ with $F$-copies, for which we will use Theorem~\ref{HR}. For a maximal fractional $F$-packing into $\mathcal{G}_{n'}$, consider all subgraphs of $\mathcal{G}_{n'}$ isomorphic to $F$.  Such subgraph exists for $n>(r-1)v(F)$ as $\chi(F) \leq r-1$. By the symmetry of $\mathcal{G}_{n'}$, any edge of $\mathcal{G}_{n'}$ appears in the same number of $F$-copies, call this amount $m$. Setting $\psi=\frac{1}{m}$ for all $F$-copies gives $\nu_F^*(\mathcal{G}_{n'})=\frac{e(\mathcal{G}_{n'})}{e(F)}$ as it clearly cannot be larger.

By the Haxell-Rödl theorem (Theorem~\ref{HR}) there exists a threshold such that whenever $n$ is large enough, there is an $F$-packing into $\mathcal{G}_{n'}$ of size at least $\nu_F^*(\mathcal{G}_{n'})-\varepsilon n^2$. Observe that this packing gives a construction in which there is no multicolor $G$-copy as there is not any $G$-copy in it. So for any large enough $n$ we have

$$\ex_F(n,G)\geq\nu_F^*(\mathcal{G}_{n'})-\varepsilon n^2=\frac{e(\mathcal{G}_{n'})}{e(F)}-\varepsilon n^2 \geq \frac{\ex(n,G)}{e(F)}-2\varepsilon n^2.$$
This implies $\ex_F(n,G)\geq\frac{\ex(n,G)}{e(F)}-o(n^2)$ as $n\to\infty$.

Now we turn to prove the upper bound.
Let $k=\ex_F(n,G)$ and assume that $\mathcal{G}_{n}$ is an extremal construction built up by $k$ pieces of disjoint $F$-copies. Suppose by contradiction that there exists a $c>0$ such that for large $n$ we have $e(\mathcal{G}_{n})=k\cdot e(F)>\left(1-\frac{1}{r-1}+c\right)\displaystyle\binom{n}{2}$. Using  Theorem~\ref{supersat} on supersaturation we get that there exists a $c'>0$ such that $\mathcal{G}_{n}$ contains at least $c'n^{v(G)}$ subgraphs isomorphic to $G$.

Let us count the ordered pairs of ($H$,$\{e,f\}$), where $H$ is a subgraph of $\mathcal{G}_{n}$ isomorphic to $G$ and $\{e,f\}$ are monochromatic edges in $H$. Call the number of such pairs $l$. By assumption each subgraph isomorphic to $G$ in $\mathcal{G}_{n}$ has at least two edges of the same color, so $l \geq c'n^{v(G)}$.

On the other hand, there are $k \displaystyle\binom{e(F)}{2}$ monochromatic edge pairs, and each such pair determines at least $3$ vertices, hence such an edge pair may be in at most $n^{v(G)-3}$ subgraphs isomorphic to $G$ in $\mathcal{G}_{n}$. Thus $l \leq k \displaystyle\binom{e(F)}2n^{v(G)-3}$. Combining the two estimations we get
$$c'n^3 \leq k \displaystyle\binom{e(F)}2.$$
Observing $n^2\geq e(\mathcal{G}_{n})=ke(F)$, we have 
$$c'n^3 \leq \frac{n^2}{e(F)}\displaystyle\binom{e(F)}2,$$
which is a contradiction for large $n$.

Thus $\ex_F(n,G)\leq\displaystyle\frac{1-\frac1{r-1}}{2e(F)}n^2+o(n^2)$ as $n\to\infty$ which finishes the proof.
\end{proof}



We prove Theorem~\ref{mainTHM} in two separate statements.

\begin{theorem}\label{haxell}
Suppose that $F$ and $G$ are graphs such that there is no homomorphism from $G$ to $F$. Then
$$\ex_F(n,G) \geq \frac{1}{v(F)^2}n^2-o(n^2).$$
\end{theorem}

\begin{proof}
Take  the $m$-blow-up $F[m]$ of $F$. Since there is no homomorphism from $G$ to $F$, this graph does not contain a subgraph isomorphic to $G$. We want a large $F$-packing into $F[m]$, for which we apply Theorem~\ref{HR} again. It is easy to give a maximal fractional $F$-packing for $F[m]$ as follows. Let $\mathcal{A}$ be the set of subgraphs spanned by one vertex from each classes of $F[m]$. Clearly, all subgraphs of $\mathcal{A}$ are isomorphic to $F$ and any edge is contained in the same number of graphs from $\mathcal{A}$. So we can put the same weight into each subgraphs from $\mathcal{A}$ and $0$ to every other subgraph such that the sum is exactly one in each edge. Consequently
$$\nu_F^*(F[m])=\frac{e(F[m])}{e(F)}=m^2.$$
From Theorem~\ref{HR} we have $\nu_F(F[m])=m^2-o((m\cdot v(F))^2)$.
In conclusion, if $v(F) \mid n$ then 
$$\ex_F(n,G) \geq \nu_F\left(F\left[\frac{n}{v(F)}\right]\right)=\frac{n^2}{v(F)^2}-o(n^2)=\Theta(n^2),$$
and by the monotonicity of $\ex_F(n,G)$, the claim follows.
\end{proof}



\begin{theorem}\label{szemreg}
Suppose that there exists a homomorphism from $G$ to $F$. Then $\ex_F(n,G)=o(n^2)$.
\end{theorem}
\begin{proof}

 Let $t=v(G)$. To obtain a  contradiction, suppose  that there exists a $c>0$ such that for all $n_0$ there is a graph $\mathcal{G}_{n}$ on $n$ vertices with $n>n_0$, with at least $cn^2$ edges and $\mathcal{G}_{n}$ can be obtained as disjoint union of $F$-copies with no multicolor $G$. We will prove that there exists a $c'>0$ such that $\mathcal{G}_{n}$ contains at least $c'n^t$ subgraphs isomorphic to $G$. It is enough to prove this as it leads to a  contradiction in the same way we proceed in the proof of the upper bound of Theorem~\ref{mainth}.

Observe that there exists a homomorphism from $G$ to $F$ exactly if $G \subset F[s]$ for some $s>0$ and that in this case $G \subset F[t]$ also holds. We want to use the regularity lemma and the key lemma, so first we choose $d$ and $\varepsilon$ for it. Let $d=\frac{c}{e(F)}$. For  $\varepsilon$ there are $3$ conditions:
\begin{enumerate}[(1)]
    \item $0<\varepsilon<d$.
    \item $\varepsilon \leq \frac{(d-\varepsilon)^{\Delta}}{2+\Delta}$. \\
  Here $\Delta=\Delta(G)$ is the maximum degree of $G$. Note that
    $$\lim_{\varepsilon \to 0^+} \frac{(d-\varepsilon)^{\Delta}}{2+\Delta}=\frac{d^{\Delta}}{2+\Delta}>0,$$
    so for small enough $\varepsilon$ this is true.
    \item $\varepsilon<\frac{c}{3e(F)}$.\\
\end{enumerate}

Choose a small enough $\varepsilon$ such that all $3$ of the above conditions hold. Define 
$$\varepsilon_0=\frac{(d-\varepsilon){^\Delta}}{2+\Delta}$$
as in the key lemma. Let $M$ be the constant defined in the statement of the regularity lemma for $\varepsilon$. Let 
$$n_0=\frac{(t-1)M}{(1-\varepsilon)\varepsilon_0}.$$
We will see that this choice guarantees the condition $t-1 \leq \varepsilon_0 m$ in the key lemma. Let $\mathcal{G}_{n}$ be a graph as above with $n>n_0$ vertices. Use the degree form of the Szemerédi regularity lemma on $\mathcal{G}_{n}$ and use the same notations, so let $V_0, V_1, \ldots, V_k$ be the given partition with $|V_i|=m$ for $i \geq 1$, $\mathcal{G'}$ is the obtained subgraph and $\mathcal{G'}'=\mathcal{G'}\setminus V_0$. Let $R$ be the so-called reduced graph, the vertices are $V_1, V_2, \ldots, V_k$ and two of them are connected if the density between them in $\mathcal{G'}$ is positive. In particular we have $\mathcal{G'}' \subset R(m)$.

Note that by the choice of $d$ and $\varepsilon$ we have $\frac{d+3\varepsilon}{2}<\frac{c}{e(F)}$ and by the consequence stated after the regularity lemma $e(\mathcal{G'}')>e(\mathcal{G}_{n})-\frac{(d+3\varepsilon)n^2}{2}$, furthermore we know that $e(\mathcal{G}_{n})\geq cn^2$. Combining these
$$e(\mathcal{G'}')>e(\mathcal{G}_{n})-\frac{(d+3\varepsilon)n^2}{2}>e(\mathcal{G}_{n})-\frac{c}{e(F)}n^2 \geq \frac{e(F)-1}{e(F)}e(\mathcal{G}_{n}).$$
This means that at least one of the $F$-copies are entirely in $\mathcal{G'}'$. In particular there is homomorphism form $F$ to $R$. We know that there exists a homomorphism form $G$ to $F$, so by transitivity there is a homomorphism from $G$ to $R$, so $G \subset R[t]$. Observe that $\mathcal{G'}'$ is a graph that can be obtained from $R$ as we obtained $G$ from $R$ in the key lemma.

We only need to check that $t-1 \leq \varepsilon_0 m$ to be able to use the key lemma.
We know that $k \leq M$ and $|V_0| \leq \varepsilon n$ so
$$m=\frac{n-|V_0|}{k} \geq \frac{(1-\varepsilon)n}{M},$$
so using the choice of $n_0$ we have
$$t-1=\frac{n_0(1-\varepsilon)\varepsilon_0}{M} <\varepsilon_0 \cdot \frac{n(1-\varepsilon)}{M} \leq \varepsilon_0 m.$$
Every condition of the key lemma holds, so there are at least 
$$(\varepsilon_0 m)^t \geq \left( \frac{(1-\varepsilon)n\varepsilon_0}{M} \right)^t$$
subgraphs of $\mathcal{G'}'$ isomorphic to $G$, so there are at least this amount of $G$-s in $\mathcal{G}_{n}$, therefore the theorem is proved with $c'=\left( \frac{(1-\varepsilon)\varepsilon_0}{M} \right)^t$.

\end{proof}

\begin{proof}[Proof of Theorem~\ref{mainTHM}]
 Theorem~\ref{mainTHM} clearly follows from Theorem \ref{haxell} and \ref{szemreg}.
\end{proof}

\section{Results on specific graph pairs}

The cases when $F$ or $G$  is a complete graph, a biclique, a path or a cycle may be of interest as more precise results can be obtained. First we consider the case when $G$ is a star $K_{1,s}$ and $F$ is an arbitrary graph. Observe that the exact formula depends only on the number of vertices of $F$, and independent of the edge structure.

\begin{prop}\label{csillag}
If $n$ is large enough, then
$$\ex_F(n,K_{1,s})=\left\lfloor \frac{n(s-1)}{v(F)} \right\rfloor.$$
\end{prop}

\begin{proof}
For $n=v(F)^{s-1}$  a $v(F)$-uniform $(s-1)$-regular hypergraph on $n$ vertices can be defined as follows. Let $V=\{(a_1,\ldots,a_{s-1}): 1\leq a_i\leq v(F) ~\forall i\in [1,  s-1]\}$ and let $v(F)$  vertices form a hyperedge if all but one of their coordinates coincide. Putting an $F$-copy on the underlying set of each hyperedge gives the statement for this value of $n$.

For $n=m\cdot v(F)^{s-1}+r$ with $m\geq v(F)^{s-1}(s-1)$ and $v(F)^{s-1}>r\geq0$, we add $m$ (disjoint) copies of the above construction to $m\cdot v(F)^{s-1}$ vertices, then delete an $F$-copy from $r(s-1)$  constructions. This leaves us with $r$ vertices appearing in no $F$-copies, $r(s-1)v(F)$  vertices in $(s-2)$ $F$-copies and the rest of the vertices in $(s-1)$ $F$-copies. Then we add $F$ copies so that each of them has exactly one of the $r$ vertices and $v(F)-1$ of the other vertices. Thus all vertices will be in $(s-1)$ $F$-copies, except for $r(s-1)$ of them being in $(s-2)$ $F$-copies. We can add arbitrary (vertex-disjoint) $F$-copies to those until less than $v(F)$ remain. This construction gives the lower bound for $n\geq v(F)^{2s-2}(s-1)$.

Conversely, if there is a vertex with $s$ colours meeting there, then there is a multicolor $K_{1,s}$. Thus in a construction for the multicolor Turán problem for $K_{1,s}$, every vertex is in at most $(s-1)$ $F$-copies, which proves that the given bound is sharp.
\end{proof}

Proposition \ref{simplelower}, which gives a simple lower bound in terms of $\ex(n,F)$ and $\ex(n,G)$, is meaningless if $\ex(n,F)>\ex(n,G)$. Thus it is worth noting some examples concerning the case $G=P_t$ in such circumstances.
A natural generalisation of the extremal structure due to Erdős and Gallai, which is based on disjoint small cliques, seems a good candidate as an extremal structure for some graphs  $F$ while very different structures provide better bounds for other graphs $F$.


\begin{prop}\label{ut} Let $F$ be a nonempty graph. Then we have
\begin{enumerate}[(i)]
    \item  $\ex_{F}(n, P_3)\sim \frac{n}{v(F)}$
    \item  $\ex_{F}(n, P_4)\geq \left\lfloor \frac{n-\alpha(F)}{v(F)-\alpha(F)} \right\rfloor.$
    \item Let $q$ be the smallest number such that $K_q$ has $t-2$ edge-disjoint  copies of $F$. Then  $\ex_{F}(n, P_t)\geq \lfloor\frac{n}{q}\rfloor(t-2)$.
\end{enumerate}

\end{prop}

\begin{proof}
 $\ex_{F}(n, P_3)$ was discussed in Proposition \ref{csillag}. \\
 For $(ii)$, identify a maximal independent set $S$ of vertices of $F$ and consider $F$-copies whose vertex sets' intersection is the $S$  in each copy of $F$.\\
Finally to prove (iii), take $\lfloor\frac{n}{q}\rfloor$ sets of size $q$ and choose $t-2$ edge-disjoint copies of $F$ is each set. This clearly provides a construction without a multicolor $P_t$.
\end{proof}

If we apply (iii) of Proposition \ref{ut} for $F=K_r$ and $F=P_r$ and take into account the path decomposition result of Tarsi \cite{Tarsi} we get in turn the corollary below.

\begin{cor} Suppose that  $r\geq t-3$. Then
$$\ex_{K_r}(n, P_t)\geq (t-2)\left\lfloor\frac{n}{(t-2)(r-(t-1)/2+1)}\right\rfloor.$$

Suppose that  $r,t\geq 3$. Then
$$\ex_{P_r}(n, P_t)\geq
\begin{cases}
    (t-2)\left\lfloor\frac{n}{r}\right\rfloor, & \text{if } r/2> t-2\\
    (t-2)\left\lfloor\frac{n}{\lceil\sqrt{2(t-2)(r-1)}+1\rceil}\right\rfloor ,              & \text{otherwise}.
\end{cases}
$$
\end{cor}

\subsection{Connection to linear Turán number of r-uniform hypergraphs}

A hypergraph $\mathcal{H}$ is a Berge copy of a graph $G$, or Berge-$G$ for short, if  we can choose a subset of size two of each hyperedge of $\mathcal{H}$ to obtain a copy of $G$. A hypergraph $\mathcal{H}$ is {\it Berge-$G$-free} if it does not contain a subhypergraph which is a Berge copy of $G$. Determining the the maximum number of hyperedges $\ex_r(n,G)$ in $r$-uniform hypergraphs which are { Berge-$G$-free} is an intensively studied variant of the original Turán-type problems, see e.g. \cite{GerbnerPatkos, GyoriL, Tait} and the references therein. A natural variant, which is also of considerable interest, is the case when the maximum number of hyperedges $\ex_r^{lin}(n,G)$ is determined in the family of {\em linear} $r$-uniform  hypergraphs, where the intersection of the hyperedges is of size at most one. In what follows, we describe its  connection to multicolor Turán problems and present the consequences of the  most notable results on  $\ex_r^{lin}(n,G)$ for various graphs $G$.

The problem of finding the maximal number of monocolored complete graphs on $r$ vertices in a rainbow-$G$-free multicolor $K_r$-packing is equivalent to finding the maximal number of hyperedges in an $r$-uniform linear hypergraph with no Berge-$G$. So we have

\begin{prop}
$$\ex_{K_r}(n,G)=\ex_r^{lin}(n,G).$$ 
\end{prop}
There are several improvements in the latter topic, which imply the following results.

\begin{result}[Ergemlidze, Győri and Methuku \cite{Gyori1}]

$$\ex_{K_3}(n,C_4) = \frac{1}{6}n^{3/2} + O(n).$$ and 
$$\ex_{K_3}(n,C_5) = \frac{1}{3\sqrt{3}}n^{3/2} + O(n)$$ 
\end{result}

\begin{result}[Füredi, Özkahya \cite{Furedi-Ozkahya}]
$$\ex_{K_3}(n,C_{2k+1}) \leq 2kn^{1+1/k}+9kn$$
\end{result}

\begin{result}[Collier-Cartaino, Graber and Jiang \cite{Jiang}]
$$\ex_{K_3}(n,C_k) = O \left(n^{1+\frac{1}{\left \lfloor\frac{k}{2}\right \rfloor}}\right).$$
\end{result}
\begin{result}[Gerbner, Methuku and Vizer \cite{Gerbner}]
  $$\ex_{K_r}(n,K_{2,t}) = (1-o_t(1))\frac{\sqrt{(t-1)}}{6}n^{3/2}.$$
\end{result}
For some generalizations, see \cite{gao}.

Following the footsteps of Erdős and Hajnal, Rödl with Phelphs \cite{Rodl1} and Sinajova \cite{Rodl2} investigated the minimum size of a maximal independent set in partial Steiner triple systems and in partial Steiner $r$-systems in general. (See also \cite{Rodl, Furedi}). They proved 

\begin{theorem}\label{stsrodl}
The minimum size of a maximal independent set in  Steiner triple systems and in partial Steiner systems of order $n$ is $\Theta(\sqrt{n\log{n}})$.
\end{theorem}

This result enables us to prove exact results for multicolor Turán numbers in the case $F=K_3$ and $G$ is a large enough clique.

\begin{prop}\label{sts-cor} There exists a constant $c>0$ such that 
 $\ex_{K_3}(n,K_t)=  \frac{1}{3}\binom{n}{2}$ provided that $t>c\sqrt{n\log{n}})$ and $n\equiv 1,3 \pmod 6$. 
\end{prop}

\begin{proof} Since $n\equiv 1,3 \pmod 6$, $n$ is an admissible order of a Steiner triple system. Theorem \ref{stsrodl} implies that if $t>c\sqrt{n\log{n}})$ for a well-chosen constant $c$, then there is a Steiner triple system with no independent set of size $t$. Colour the triples with different colours and take their (coloured) shadow, i.e. the triangles formed by the edges corresponding to the triples. Then there is no multicolor $K_t$ is the obtained coloured graph, hence the proof.
\end{proof}
 Note that for $n\equiv 5 \pmod 6$ and $t$ is large enough, a constant size error term appears compared to the case $n\equiv 1,3 \pmod 6$ as almost all edges can be covered by disjoint triangles; while the error term will be linear for even $n$ since at least one edge incident to each vertex will not be covered by triangles. 

\section{Concluding remarks and open problems}

In this section we pose some open problems and discuss their relevance and possible applications.

\begin{problem}
Determine the set of pairs $F, G$ for which  $\ex_F{(n,G)}\sim \frac{\ex{(n,G)}}{e(F)}$.
\end{problem}
According to a fundamental theorem of Wilson \cite{Wilson}, for all sufficiently large $n$, the complete graph $K_n$  has an $F$-decomposition subject to the divisibility conditions concerning the number of edges $F$ and $K_n$ and the greatest common divisor of the degrees of $F$. This suggests that for almost all graphs $F$, the multicolor Turán number might be $\frac{\ex{(n,G)}}{e(F)}$ if $G$ is a large enough tree, as in this case the (almost) extremal graphs consists of disjoint large cliques of size $v(G)-1$.\\
We also conjecture that $\ex_{K_{1,2}}{(n,C_4)}=(1+o(1))\frac{\ex{(n,C_4)}}{2}$ holds. If true, this would in turn imply an improvement in the proof of Ergemlidze, Győri,  Methuku and Salia on the generalized Turán number $\ex(n, C_3, C_5)$ \cite{Gyori}. Indeed, at some point in their proof they argue that erasing an edge from each element of a set of edge-disjoint cherries yields a $C_4$-free graph, hence the multicolor Turán number of $C_4$ applies. In fact it would be desirable to get an upper bound on the number of cherries not creating a multicolor quadrilateral supposing that the set if cherries does not create any $5$-cycles. We believe that under this condition,  extremal constructions may be gained  essentially the same way as  described by Bollobás and Győri \cite{Boll-Gyori}, which can be viewed as a blow-up of the incidence graph of a projective plane by doubling each vertex corresponding to a line in the projective plane. Here, the image of an edge becomes a cherry ($K_{1,2}$).
Note that since their short and neat proof on the upper bound, some improvements have been established \cite{Beka}.

\begin{problem}

Determine the set of pairs $F, G$ for which  $\ex_F{(n,G)}\sim \frac{1}{v(F)^2}n^2$.
\end{problem}

Note that these are exactly those graphs where the lower bound of Theorem \ref{haxell} is asymptotically sharp. We believe that the pair $F=C_5, G=C_3$ serves as an example for this class of pairs, thus the extremal construction is essentially the same as the answer for the famous Erdős pentagon problem which was resolved using flag algebra calculus \cite{Grzesik, Kral}.

\begin{problem}
Determine the set of pairs $F, G$ for which $\ex_F{(n,G)}= \Theta( \ex{(n,G)})$.
\end{problem}

Concerning the multicolor Turán number of cliques, the order of magnitude is an open question if $F=G$ while for cliques  with $F\subset G$, an exact result for infinitely many $n$ is only known for $G$ being large in view of Proposition \ref{sts-cor}. A natural question thus raises as follows.

\begin{problem}
For fixed $r$, determine the least $k=f(r)$ such that $\ex_{K_r}(n, K_k)$ is  equal to $\binom{n}{2} /\binom{r}{2}-~O(n)$. 
\end{problem}

It would be also interesting to obtain lower bounds for  $\ex_{K_r}(n, K_k)$ with $r>k$ in the form of $n^{2-o(1)}$ which would give a generalization of the result on the (6,3)-problem. We intend to work on further results in this area.




\end{document}